\documentclass[12pt]{amsart}
\usepackage{amssymb, enumerate, mdwlist, MnSymbol}

\evensidemargin\paperwidth
\advance\evensidemargin-\textwidth
\advance\oddsidemargin-1in
\evensidemargin\oddsidemargin

\topmargin\paperheight
\advance\topmargin-\textheight
\advance\topmargin-1in

\theoremstyle{plain}
\newtheorem{theorem}{Theorem}[section]
\theoremstyle{plain}
\newtheorem{proposition}[theorem]{Proposition}
\theoremstyle{plain}
\newtheorem{lemma}[theorem]{Lemma}
\theoremstyle{plain}
\newtheorem{corollary}[theorem]{Corollary}
\theoremstyle{plain}

\theoremstyle{plain}

\theoremstyle{plain}
\newtheorem{conjecture}[theorem]{Conjecture}
\theoremstyle{plain}

\theoremstyle{plain}

\theoremstyle{definition}
\newtheorem{definition}[theorem]{Definition}
\theoremstyle{remark}
\newtheorem{remark}[theorem]{Remark}
\theoremstyle{remark}

\theoremstyle{definition}

\theoremstyle{remark}

\title[Extreme counterexample]{An extreme counterexample to the Lubotzky--Weiss conjecture.}
\author{Masato Mimura}
\address{Masato Mimura\\
Mathematical Institute, Tohoku University, Japan}
\email{mimura-mas@m.tohoku.ac.jp}
\date{\today}

\begin{document}

\begin{abstract}
In 1993, Lubotzky and Weiss conjectured that if a compact group admits two finitely generated dense subgroups, one of which is amenable and the other has Kazhdan's property $(\mathrm{T})$, then it would be finite. This conjecture was resolved in the negative by Ershov and Jaikin-Zapirain, and by Kassabov around 2010. In the present paper, we provide an extreme counterexample to this conjecture. More precisely, the latter dense group with property $(\mathrm{T})$ may contain a given countable residually finite group; in particular, it can be non-exact  by a result of Osajda. We may construct these counterexamples with a compact group common for all countable residually finite groups.
\end{abstract}


\maketitle


\section{Introduction}
From a study of expander graphs, Lubotzky and Weiss \cite[Conjecture~5.4]{LubotzkyWeiss} proposed the following conjecture; we refer the reader to \cite{BookBekkadelaHarpeValette} as a comprehensive treatise on amenability and property $(\mathrm{T})$. In broad strokes, amenability and property $(\mathrm{T})$ are \textit{opposite} for infinite discrete groups. Property $(\mathrm{T})$ forces group actions to have spectral gaps. Amenability is characterized by the existence of a F\o lner sequence; this intrinsic condition on a group, in principle, serves as an obstruction for a group action to having a spectral gap.

\begin{conjecture}[Lubotzky--Weiss conjecture]\label{conjecture=LubotzkyWeiss}
Let $K$ be a compact group. If $\Lambda_1$ and $\Lambda_2$ are both finitely generated subgroups dense in $K$ with $\Lambda_1$ amenable and $\Lambda_2$ having Kazhdan's property $(\mathrm{T})$, then $K$ is finite.
\end{conjecture}

This conjecture has been resolved \textit{in the negative} by Ershov and Jaikin-Zapirain, and independently by Kassabov; see \cite[Subsection~6.3]{ErshovJaikinZapirain} and Section~\ref{section=Common} in the present paper. The counterexample in \cite{ErshovJaikinZapirain} is of the following form for a fixed prime $p$:
\[
K=\prod_{m\in \mathbb{N}_{\geq 1}}\mathrm{SL}(3m,\mathbb{F}_p),
\]
which is equipped with the product topology.

This paper provides the following \textit{extreme counterexample} to Conjecture~\ref{conjecture=LubotzkyWeiss}; see Remark~\ref{remark=2-generated} for a further strengthening in terms of numbers of generators.

\begin{theorem}[Main Theorem]\label{theorem=Main}
Given a countable residually finite group $H$, there exist a compact group $K=K^{(H)}$ and two finitely generated subgroups $\Lambda_1$ and $\Lambda_2$ both dense in $K^{(H)}$ with the following properties:
\begin{enumerate}[$(i)$]
  \item $\Lambda_1$ is amenable $($see Remark~$\ref{remark=LFNF-lifts}$ for more information$)$;
  \item $\Lambda_2$ has property $(\mathrm{T})$ $($see Remark~$\ref{remark=Upgrading}$ for some stronger properties$)$; and 
  \item $\Lambda_2$ contains  an isomorphic copy of $H$.
\end{enumerate}
\end{theorem}

Recall that a countable group $G$ is \textit{residually finite} if there exists a \textit{chain} $(N_m)_{m\in \mathbb{N}}$, namely $N_{m+1}\leqslant N_{m}$ for every $m$, of finite index normal subgroups of $G$ such that $\bigcap_{m}N_m=\{e_G\}$. 

\begin{remark}\label{remark=Malcev}
The two conditions on $H$ as in Theorem~\ref{theorem=Main} are both necessary: it is clear that every subgroup of a finitely generated group is countable. By Mal'cev's result, every subgroup of finitely generated subgroup of a compact group is residually finite. Indeed, by the Peter--Weyl theorem, compact groups are residually linear.
\end{remark}

Our construction provide the compact group $K^{(H)}$ as in Theorem~\ref{theorem=Main}  which is of the following form:
\[
K^{(H)}=\prod_{m\in \mathbb{N}}\mathrm{SL}(nl_m,\mathbb{F}_p)
\]
for a fixed odd prime $p$ and for fixed $n\in \mathbb{N}_{\geq 3}$. Here,  $(l_m)_{m\in \mathbb{N}}$ is a certain strictly increasing sequence of integers, depending on $H$ and on the choice of $p$; see Section~\ref{section=TheProof} for more details.

Theorem~\ref{theorem=Main} may be, for instance, combined with a result of Osajda \cite{OsajdaRF} to obtain the following notable counterexample to Conjecture~\ref{conjecture=LubotzkyWeiss}.

\begin{corollary}\label{corollary=Nonexact(T)}
There exist a compact group $K$ and two finitely generated subgroups $\Lambda_1$ and $\Lambda_2$ both dense in $K$ with $\Lambda_1$ amenable, and $\Lambda_2$ having Kazhdan's property $(\mathrm{T})$ and being $\mathrm{non}$-$\mathrm{exact}$.
\end{corollary}

Here \textit{exactness} for countable groups is equivalent to \textit{Yu's property A}, which may be regarded as the counterpart of amenability of groups in coarse geometry; see \cite{OzawaExactness} and \cite{bookNowakYu}. Non-exactness of groups is regarded as a pathological property. For instance, every countable subgroup of $\mathrm{GL}(n,A)$ for $n\in \mathbb{N}_{\geq 2}$ and for a unital \textit{commutative} (associative) ring $A$ is exact; see \cite[Theorem~5.2.2 and Theorem~4.6]{guentnertesserayu}. 

\begin{proof}[Proof of Corollary~$\ref{corollary=Nonexact(T)}$ modulo Theorem~$\ref{theorem=Main}$]
Osajda \cite{OsajdaRF} constructed a residually finite (finitely generated) group that is non-exact. Apply Theorem~\ref{theorem=Main} with $H$ being that group; note that exactness passes to subgroups.
\end{proof}

We explain two motivations of Theorem~\ref{theorem=Main}:
\begin{enumerate}[$(a)$]
  \item From a finitely generated subgroup $\Lambda$ dense in a compact group $K$, the \textit{compact action} $\Lambda\curvearrowright K$ is constructed; this action is free and minimal. With respect to the Haar measure on $K$, it is measure-preserving and ergodic; see \cite[Remark~3.5]{MimuraRF}. Thus, by Theorem~\ref{theorem=Main}, we may obtain two compact actions, $\Lambda_1\curvearrowright K$ and $\Lambda_2\curvearrowright K$, on the common set $K$ with extremely contrastive properties. 
  \item By Remark~\ref{remark=Malcev}, the group $\Lambda_2$ as in Theorem~\ref{theorem=Main} is residually finite. It provides us with a \textit{residually finite} group having property $(\mathrm{T})$ \textit{with a specified property}, such as non-exactness. Residual finiteness plays a fundamental role in study of profinite groups and profinite actions. Further, from a (finitely generated) residually finite group, we can construct a \textit{box space}, which serves as a powerful device to construct examples of metric spaces with noteworthy coarse geometric properties; see the introduction of \cite{OsajdaRF} and \cite[4.4]{bookNowakYu}.
\end{enumerate}
As an application of $(b)$, we may construct expanders with \textit{geometric property $(\mathrm{T})$} from a non-exact group; compare with \cite[Theorem 1.1.(4)]{WillettYu}.

Furthermore, we may take a compact group $L$, \textit{common for all countable and residually finite $H$}, to accommodate finitely generated dense gorups, some of which is amenable and some of which has property $(\mathrm{T})$ with a subgroup isomorphic to $H$.

\begin{theorem}[Variety of finitely generated dense subgroups of a \textit{common} compact group]\label{theorem=Common}
Let $p$ be an odd prime and let $n\in \mathbb{N}_{\geq 3}$. Then for the compact group
\[
\mathcal{K}(=\mathcal{K}_{n,p})=\prod_{i\in\mathbb{N}_{\geq 1}}\mathrm{SL}(ni,\mathbb{F}_p)
\]
$($equipped with the product topology$)$, the following hold true:
\begin{enumerate}[$(1)$]
  \item There exists $\Sigma_1\leqslant \mathcal{K}$, a finitely generated dense subgroup of $\mathcal{K}$ that is amenable $($see Remark~$\ref{remark=LFNF-lifts}$ for more information$)$.
  \item For every countable and residually finite group $H$, there exists $\Sigma_2=\Sigma_2^{(H)}\leqslant \mathcal{K}$, a dense subgroup of $\mathcal{K}$, such that it is generated by $13$ elements $($as a group$)$, it has property $(\mathrm{T})$ and that it contains an isomorphic copy of $H$.
\end{enumerate}
\end{theorem}
We may modify $\mathcal{K}$ to obtain \textit{$2$-generated} $\Sigma_1$ and $\Sigma_2^{(H)}$; see Theorem~\ref{theorem=2-generatedCommon}.

On point $(a)$ mentioned above, in our constructions out of Theorems~\ref{theorem=Main} and \ref{theorem=Common}, we have two \textit{profinite actions}; see \cite{AbertElek}. Here we discuss the case associated with Theorem~\ref{theorem=Common}. The group $\mathcal{K}$ is homeomorphic to the Cantor set as a topological space; in fact, the two profinite actions $\Sigma_1\curvearrowright \mathcal{K}$ and $\Sigma_2^{(H)}\curvearrowright \mathcal{K}$ appearing above are given by projective systems with a common sequence of finite groups
\[
\Sigma_1\curvearrowright \mathcal{K}=\varprojlim_{j} (\Sigma_1\curvearrowright \prod_{i\in \mathbb{N}_{\leq j}}\mathrm{SL}(ni,\mathbb{F}_p)),\quad \Sigma_2^{(H)}\curvearrowright \mathcal{K}=\varprojlim_{j} (\Sigma_2^{(H)}\curvearrowright \prod_{i\in \mathbb{N}_{\leq j}}\mathrm{SL}(ni,\mathbb{F}_p)).
\]
Furthermore, if we let $\Sigma_1$ act from the left and let $\Sigma_2^{(H)}$ act from the right, then these two (free, minimal and ergodic) actions $\Sigma_1\curvearrowright \mathcal{K}\curvearrowleft \Sigma_2^{(H)}$ \textit{commute}.

The proof of Theorem~\ref{theorem=Main} employs \textit{LEF approximations} of a  group; see Definition~\ref{definition=LEF}. The \textit{LEF} (Locally Embeddable into Finite groups) property is a weak form of residual finiteness, and it is described in terms of convergence in \textit{the space of marked groups}; by this weakening, we have room to construct two markings of a certain sequence of finite groups with respect to which the two limit groups have contrasting group properties. This method has been studied and developed by the author \cite{MimuraRF}; we will briefly recall concepts related to the above in Section~\ref{section=TheProof}. Prior to that, in Section~\ref{section=RF(T)}, we prove Proposition~\ref{proposition=RF(T)}, which may be seen as an intermediate step to Theorem~\ref{theorem=Main} from the viewpoint of $(b)$ above. In Section~\ref{section=Common}, we demonstrate Theorem~\ref{theorem=Common}. Section~\ref{section=Remarks} is a collection of remarks; there we discuss Theorems~\ref{theorem=2-generated} and \ref{theorem=2-generatedCommon}.

In this paper, for a prime $p$, let $\mathbb{F}_p$ denote the finite field of order $p$. For $n\in \mathbb{N}_{\geq 1}$, let $[n]$ denote the set $\{1,2,\ldots ,n\}$. Rings are assumed to be associative; we exclude the zero ring from unital rings.

\section{Residually finite group with property $(\mathrm{T})$ containing a given residually finite group}\label{section=RF(T)}

In this section, we prove the following.
\begin{proposition}\label{proposition=RF(T)}
Given countable and residually finite group $H$, there exists a $\mathrm{residually}$ $\mathrm{finite}$ group $G$ with property $(\mathrm{T})$ that contains an isomorphic copy of $H$.
\end{proposition}
As we saw above, this proposition may be seen as a step towards Theorem~\ref{theorem=Main}; recall Remark~\ref{remark=Malcev}. In particular, by \cite{OsajdaRF}, it follows that there exists a residually finite group with property $(\mathrm{T})$ that is \textit{non-exact}.

Our construction of $G$ as in Proposition~\ref{proposition=RF(T)} employs \textit{elementary groups} over a \textit{non-commutative} ring; compare with the result of \cite{guentnertesserayu} mentioned in the introduction. Let $R$ be a unital ring and $n\in \mathbb{N}_{\geq 2}$. Then the \textit{elementary group} $\mathrm{E}(n,R)$ of degree $n$ over $R$ is defined as the subgroup of $\mathrm{GL}(n,R)$ generated by elementary matrices $e_{i,j}^r$, $i\ne j\in [n]$, $r\in R$. Here for $k,l\in [n]$, $(e_{i,j}^r)_{k,l}$ equals $1$ if $k=l$, $r$ if $(k,l)=(i,j)$, and $0$ otherwise. A key to the proof of Proposition~\ref{proposition=RF(T)} is the following:

\begin{lemma}\label{lemma=Order2}
Let $R$ be a unital ring. Let $\Omega\ne \emptyset$ be a subset of the multiplicative group $R^{\times}$. Assume that for every $\omega \in \Omega$, $\omega^2=1$ holds true. Then for every $n\in \mathbb{N}_{\geq 2}$, the elementary group $\mathrm{E}(2,R)$ contains an isomorphic copy of $\langle \Omega\rangle (\leqslant R^{\times})$.
\end{lemma}

\begin{proof}
Since $\mathrm{E}(n,R)\hookrightarrow \mathrm{E}(n+1,R)$ for every $n$ and $R$, we prove the assertion only for $n=2$. For $r_1,r_2\in R^{\times}$, we write the element $\left(
\begin{array}{cc}
r_1 & 0\\
0 & r_2
\end{array}\right) (\in \mathrm{GL}(2,R))$ as $D(r_1,r_2)$. For $r\in R^{\times}$, we have that
\[
D(r_,r^{-1})=e_{1,2}^r e_{2,1}^{-(r^{-1})}e_{1,2}^r e_{1,2}^{-1}e_{2,1}^{1}e_{1,2}^{-1};
\]
hence it belongs to $\mathrm{E}(2,R)$. Let $D_{\Omega}=\{D(\omega,\omega^{-1}):\omega \in \Omega\}$ and $Z$ be the subgroup of $\mathrm{E}(2,R)$ generated by $D_{\Omega}$. Since every $\omega\in \Omega$ satisfies that  $\omega=\omega^{-1}$, the map $\langle \Omega\rangle \ni \lambda \mapsto D(\lambda,\lambda)\in Z $ gives a group isomorphism.
\end{proof}

We also recall the following lemma from  \cite[Lemma~5.1 and Remark~5.2]{MimuraRF}.

\begin{lemma}[\cite{MimuraRF}]\label{lemma=RFOrder2}
Let $k\in \mathbb{N}_{\geq 1}$. Let $H$ be a residually finite group that is $k$-generated. Then, there exists a residually finite group with generating set $\Xi$  such that it contains an isomorphic copy of $H$, $\# \Xi=2k$, and that every element in $\Xi$ is of order $2$.
\end{lemma}

The following theorem is one of the main theorem of the celebrated paper of  Ershov and Jaikin-Zapirain \cite[Theorem 1.1]{ErshovJaikinZapirain}. See also \cite{MimuraExpository} for an alternative short proof of this qualitative  statement of Theorem~\ref{theorem=ErshovJaikin}.

\begin{theorem}[Ershov and Jaikin-Zapirain, \cite{ErshovJaikinZapirain}]\label{theorem=ErshovJaikin}
Let $R$ be a unital and finitely generated $($associative$)$ ring. Let $n\in \mathbb{N}_{\geq 3}$. Then the elementary group $\mathrm{E}(n,R)$ has property $(\mathrm{T})$.
\end{theorem}

Hereafter, for a ring $A$ and a group $G$ let $A[G]$ denote the group ring of $G$ over $A$.

\begin{proof}[Proof of Proposition~$\ref{proposition=RF(T)}$]
Let $H$ be a countable residually finite group. By Wilson's theorem \cite{Wilson}, $H$ embeds into a residually finite group $\overline{H}$ that is $2$-generated. Apply Lemma~\ref{lemma=RFOrder2}. Then again by applying \cite{Wilson}, we finally obtain a residually finite $2$-generated group $\tilde{H}$ that admits $\Xi\subseteq \tilde{H}$ with $\# \Xi=4$ such that $\langle \Xi\rangle $ contains an isomorphic copy of $\overline{H}$ and every $\xi\in \Xi$ is of order $2$. 

We claim that for every $n\in \mathbb{N}_{\geq 3}$ and for every prime $p$, the group $G=\mathrm{E}(n,\mathbb{F}_{p}[\tilde{H}])$ fulfills all conditions of  Theorem~\ref{theorem=Main}. Indeed, by Lemma~\ref{lemma=Order2} with $\Omega =\Xi$, we see that the $\Lambda$ above contains an isomorphic copy of $\langle \Xi\rangle$. Here we naturally embed $\tilde{H}$ into $\mathbb{F}_p[\tilde{H}]$. Since $\langle \Xi\rangle\simeq \tilde{H}\geqslant H$, $G$ contains an isomorphic copy of $H$. Next, we will show that $G$ is residually finite. Take a chain $(\tilde{N}_m)_m$, as in the introduction, of normal subgroups of the residually finite $\tilde{H}$. For every $m\in\mathbb{N}$, let $\pi_m\colon \tilde{H}\twoheadrightarrow \tilde{H}/\tilde{N}_m$ be the natural projection. Then the map $\mathbb{F}_p[\tilde{H}]\ni g\mapsto \pi_m(g)\in \mathbb{F}_p[\tilde{H}/\tilde{N}_m]$ induces
\[
\rho_m\colon G\twoheadrightarrow \mathrm{E}(n,\mathbb{F}_p[\tilde{H}/\tilde{N}_m]) .
\]
Then $(\mathrm{Ker}(\rho_m))_{m\in\mathbb{N}}$ provides a desired chain of normal subgroups of $G$. 

Finally, by Theorem~\ref{theorem=ErshovJaikin}, $G$ has property $(\mathrm{T})$.
\end{proof}

In the proof above, the field $\mathbb{F}_p$ may be replaced with several other rings; for instance, an arbitrary  unital finite ring, including all finite fields, and the ring $\mathbb{Z}$.

\section{The proof of Theorem~$\ref{theorem=Main}$}\label{section=TheProof}
The proof of Theorem~\ref{theorem=Main} employs convergences in the space of marked groups, specially the LEF property. We briefly recall them; see \cite{MimuraRF} and \cite{MimuraSakoPartI}, and references therein for more details. 

For a fixed $k\in \mathbb{N}_{\geq 1}$, $\mathbf{G}=(G;s_1,\ldots ,s_k)=(G;S)$ is a \textrm{$k$-marked group} if $G$ is a group and $S=(s_1,\ldots ,s_k)$ is an ordered $k$-tuple of generators of $G$ (as a group). Such an $S$ is called a (\textit{$k$-})\textit{marking} of $G$. (We identify two isomorphic $k$-marked groups.) The set of all $k$-marked groups $\mathcal{G}(k)$, the \textit{space of $k$-marked groups}, is naturally equipped with a compact metrizable topology, as we describe below: $(\mathbf{G}_m)_{m\in \mathbb{N}}=((G_m;s_1^{(m)},\ldots ,s_k^{(m)}))_{m}$, a sequence in $\mathcal{G}(k)$ converges to $\mathbf{G}_{\infty}=(G_{\infty};s_1^{(\infty)},\ldots ,s_k^{(\infty)})$ if the following holds true. For every $R\geq 0$, there exists $m_R\in \mathbb{N}$ such that for all $m\in \mathbb{N}_{\geq m_R}$, the map sending $s_j^{(m)}$ to $s_j^{(\infty)}$ is a partial isomorphism from the $B_{\mathbf{G}_m}(e_{G_m},R)$ to $B_{\mathbf{G}_{\infty}}(e_{G_{\infty}},R)$. Here for $m\in \mathbb{N}\cup \{\infty\}$, $B_{\mathbf{G}_m}(e_{G_m},R)$ denotes the ball centered at $e_{G_m}$ of radius $R$ with respect to the word length on $G_m$ with respect to the marking $(s_1^{(m)},\ldots ,s_k^{(m)})$; for non-empty subsets $A\subseteq \Gamma_1$ and $B\subseteq \Gamma_2$, a map $\phi\colon A\to B$ is a \textit{partial homomorphism} if for every $\gamma,\gamma'\in A$ \textit{with} $\gamma\gamma'\in A$, 
\[
\phi(\gamma\gamma')=\phi(\gamma)\phi(\gamma')
\]
holds true.
A bijective partial homomorphism is called a \textit{partial isomorphism}. The convergence explained above determines the topology on $\mathcal{G}(k)$, which is called the \textit{Cayley topology} in some literature. We write this convergence as $\mathbf{G}_m\stackrel{\mathrm{Cay}}{\longrightarrow} \mathbf{G}_{\infty}$.

\begin{definition}\label{definition=LEF}
\begin{enumerate}[$(1)$]
\item A finitely generated group $\Gamma$ is said to be \textit{LEF} (\textit{Locally Embeddable into Finite groups}) if for some (equivalently, every) marking $S_{\infty}$ of $\Gamma$, there exists a sequence $(\mathbf{G}_{m})_{m\in \mathbb{N}}=((G_m;S_m))_m$ such that it consists of finite marked groups (that means, $G_m$ is finite for every $m\in \mathbb{N}$) and that  $\mathbf{G}_{m}\stackrel{\mathrm{Cay}}{\longrightarrow}(\Gamma;S_{\infty})$. 
\item A sequence $(\mathbf{G}_m)_m$ with the two properties as in $(1)$ is called a \textit{LEF approximation} of $(\Gamma;S_{\infty})$.
\end{enumerate}
\end{definition}
The LEF property is strictly weaker than residual finiteness; see Remark~\ref{remark=LEF}. More precisely, for a LEF approximation, $(G_m;s_1^{(m)},\ldots ,s_k^{(m)})$ need not be a \textit{marked group quotient} of $(\Gamma,s_1^{(\infty)},\ldots ,s_k^{(\infty)})$, where being a marked group quotient means that the map $s_j^{(\infty)}\mapsto s_j^{(m)}$ for each $j\in [k]$ extends to a group homomorphism. Nevertheless, via \textit{diagonal products}, we can construct a \textit{residually finite} group out of a LEF approximation $(\mathbf{G}_m)_m$ of $(\Gamma;S_{\infty})$, as we will see below.

\begin{definition}\label{definition=DiagonalProduct}
Let $(\mathbf{G}_m)_{m\in \mathbb{N}}=((G_m;s_1^{(m)},\ldots ,s_k^{(m)}))_m$ be a LEF approximation of $(\Gamma;s_1^{(\infty)},\ldots ,s_k^{(\infty)})$. We define the \textit{diagonal product} $\Delta_{m\in \mathbb{N}}(\mathbf{G}_m)\in \mathcal{G}(k)$ as follows: Let $K=\prod_{m\in \mathbb{N}}G_m$ and for $j\in [k]$, let $s_j=(s_j^{(0)},s_j^{(1)},\ldots ,s_j^{(m)},\ldots  ) \in K$. Define 
\[
\Delta_{m\in \mathbb{N}}(\mathbf{G}_m)=(\Lambda;s_1,\ldots ,s_k),\quad \textrm{where}\quad L=\langle s_1,\ldots ,s_k\rangle\quad (\leqslant K).
\]
\end{definition}
Then there exists a short exact sequence
\[
1\quad \longrightarrow \quad N=\Lambda\cap \bigoplus_{m\in \mathbb{N}}G_m \quad \longrightarrow \quad  \Lambda \quad \longrightarrow \quad \Gamma \quad \longrightarrow \quad 1,\tag{$\ast$}
\]
where $\Lambda\twoheadrightarrow \Gamma$ above is extended by the map $s_j\mapsto s_j^{(\infty)}$ for each $j\in [k]$. Hence, $N$ above is locally finite (that means, for every non-empty finite subset, the subgroup generated by it is finite); in particular, $N$ is amenable. The LEF property of $\Gamma$ is upgraded to residual finiteness of $\Lambda$ by this procedure. 

In Definition~\ref{definition=DiagonalProduct}, $\Lambda$ is not necessarily dense in $K$. However, by the Goursat lemma, the following holds true; see \cite[Lemmata~4.6 and 4.7]{MimuraRF} for instance.

\begin{lemma}\label{lemma=Goursat}
We stick to the setting as in Definition~$\ref{definition=DiagonalProduct}$. Assume that $\mathrm{no}$ finite simple $($non-trivial$)$ group appears as a group quotient of $G_m$ for  more than one $m\in \mathbb{N}$. Then, the underlying group $\Lambda$ of $\Delta_{m\in \mathbb{N}}((G_m;S_m))$ is dense in $K=\prod_{m\in \mathbb{N}}G_m$. Here we equip $K$ with the product topology and regard it as a compact group.
\end{lemma}

Before starting the proof of Theorem~\ref{theorem=Main}, we sketch the rough idea. A naive idea which is derived from Proposition~\ref{proposition=RF(T)} is to focus on the sequence of finite groups 
\[
(G_m)_m=(\mathrm{E}(n,\mathbb{F}_p[\tilde{H}/\tilde{N}_m]))_m
\]
as in the proof of Proposition~\ref{proposition=RF(T)}. Since $\tilde{H}$ is $2$-generated, $\mathrm{E}(n,\mathbb{F}_p[\tilde{H}])$ has a $6$-marking $S$; compare with Remark~\ref{remark=Generators}. Hence $(G_m)_{m}$ admits a system of $4$-markings $(S_m)_m=(\rho_m(S))_m$, with respect to which the underlying group of the diagonal product $\Delta_{m\in \mathbb{N}}((G_m;S_m))$ has property $(\mathrm{T})$. Then it might end the proof if we succeed in finding another system of $k'$-markings $(S'_m)_m$ (for some fixed $k'\in \mathbb{N}_{\geq 1}$), with respect to which $((G_m;S'_m))_m$ converge to an amenable marked group. 

However, here we encounter the following two difficulties: one, which is the major part, is that the group ring $\mathbb{F}_p[\tilde{H}/\tilde{N}_m]$ is \textit{far from} simple. It possesses \textit{all} pieces of information of the representations of the finite group $\tilde{H}/\tilde{N}_m$ over $\mathbb{F}_p$; it is unclear how to take a system of markings of $(G_m)_m$ \textit{with a fixed cardinality} so as to have a convergent sequence to an amenable marked group. The other difficulty is that if $p\mid\# (\tilde{H}/\tilde{N}_m)$, then it may cause certain problems coming from modular representation theory. For instance, in that case, $\mathbb{F}_p[\tilde{H}/\tilde{N}_m]$ may \textit{not} be semi-simple; moreover, it may cause some problem when we try to apply Lemma~\ref{lemma=Goursat}. For a fixed $p$, we can always arrange $(\tilde{H}/\tilde{N}_m)_m$ such that it meets the condition $p\mid\# (\tilde{H}/\tilde{N}_m)$ for all $m\in \mathbb{N}$; for instance, replace it with $((\tilde{H}/\tilde{N}_m)\times (\mathbb{Z}/p\mathbb{Z}))_m$. On the other hand, it is \textit{no way} possible, for a general sequence $(\tilde{H}/\tilde{N}_m)_m$, to modify $(\tilde{H}/\tilde{N}_m)_m$ such that for every $m\in \mathbb{N}$, $p\nmid\# (\tilde{H}/\tilde{N}_m)$ holds.

Our idea to deal with the first (and major) difficulty is to \textit{encode the convergence} $\tilde{H}/\tilde{N}_m \stackrel{\mathrm{Cay}}{\longrightarrow}  \tilde{H}$ (with respect to the markings given  by $\pi_m\colon \tilde{H}\twoheadrightarrow \tilde{H}/\tilde{N}_m$) \textit{into symmetric groups}, as follows; it was observed in \cite[Lemma~4.9]{MimuraRF}. Here for a non-empty set $B$, the (full) \textit{symmetric group} $\mathrm{Sym}(B)$ is defined as the group of all bijections on $B$; we denote by $\mathrm{Sym}_{<\aleph_0}(B)$ the group of all bijections on $B$ that fix all but finitely many elements in $B$.

\begin{lemma}[Encoding into symmetric groups: \cite{MimuraRF}]\label{lemma=SymmetricGroups}
Let $((G_m;s_1^{(m)},\ldots ,s_k^{(m)}))_{m\in \mathbb{N}}$ be a LEF approximation of $(\Gamma;s_1^{(\infty)},\ldots ,s_k^{(\infty)})$ in $\mathcal{G}(k)$. Assume that $\Gamma$ is an infinite group and for all $j\in [k]$ and all $m\in \mathbb{N}$, $s_j^{(m)}\ne e_{G_m}$. Then, in $\mathcal{G}(2k)$, we have 
\begin{align*}
&(\mathrm{Sym}(G_m);\chi_{s_1^{(m)}},\ldots ,\chi_{s_k^{(m)}},\theta_{s_1^{(m)}},\dots ,\theta_{s_k^{(m)}})\quad \\
\stackrel{\mathrm{Cay}}{\longrightarrow} \quad &(\mathrm{Sym}_{<\aleph_0}(\Gamma)\rtimes \Gamma;\chi_{s_1^{(\infty)}},\ldots ,\chi_{s_k^{(\infty)}},\theta_{s_1^{(\infty)}},\dots ,\theta_{s_k^{(\infty)}}).
\end{align*}
Here the group $\mathrm{Sym}_{<\aleph_0}(\Gamma)\rtimes \Gamma$ is constructed by the action $\mathrm{Sym}_{<\aleph_0}(\Gamma)\curvearrowleft \Gamma$ induced by right multiplication; for a group $G$ and for $g\in G\setminus \{g\}$, we define $\chi_g\in \mathrm{Sym}_{<\aleph_0}(G)$ and $\theta_g\in \mathrm{Sym}(G)$, respectively, by 
\begin{align*}
\chi_g&=\textrm{$($the transposition of $e_G$ and $g$$)$},\\
\theta_g&=\textrm{$($the permutation induced by the right multiplication of $g$$)$}.
\end{align*}
\end{lemma}

The following lemma in \cite[Lemma~4.1]{MimuraRF} is obvious.

\begin{lemma}\label{lemma=LEFsubgroups}
Let $((\overline{G}_m;\overline{s}_1^{(m)},\ldots ,\overline{s}_{\ell}^{(m)}))_{m\in \mathbb{N}}$ be a convergent sequence to $(\overline{G}_{\infty};\overline{s}_1^{(\infty)},\ldots ,\overline{s}_{\ell}^{(\infty)})$ in the Cayley topology. Let $k\in \mathbb{N}_{\geq 1}$ and let $\omega_1,\ldots, \omega_k$ be words in $F_{\ell}$. For every $m\in \mathbb{N}\cup\{\infty\}$, let $s_j^{(m)}=\omega_j(\overline{s}_1^{(m)},\ldots ,\overline{s}_{\ell}^{(m)})$ for every $j\in [k]$. Let $G_m$ be the subgroup of $\overline{G}_m$ generated by $s_j^{(m)}$, $j\in [k]$. Then $((G_m;s_1^{(m)},\ldots ,s_k^{(m)}))_{m\in \mathbb{N}}$ converges to $(G_{\infty};s_1^{(\infty)},\ldots ,s_k^{(\infty)})$ in the Cayley topology.

In particular, if we a priori know that $((G_m;s_1^{(m)},\ldots ,s_k^{(m)}))_{m\in \mathbb{N}}$ above converges to $(G;s_1,\ldots ,s_k)$ in the Cayley topology, then $\overline{G}_{\infty}$ contains an isomorphic copy of $G$.
\end{lemma}

Now we are ready to overcome the second difficulty; it is because, thanks to Lemma~\ref{lemma=SymmetricGroups}, we can \textit{switch} from (mysterious) finite groups $(G_m)_m$ to \textit{symmetric groups} $(\mathrm{Sym}(G_m))_{m}$, \textit{for which representation theory is well-developed}. In this paper, we fix a prime $p$ in order to obtain a ``small'' amenable group $\Lambda_1$; compare with Remark~\ref{remark=LFNF-lifts}. 

\begin{proof}[Proof of Theorem~$\ref{theorem=Main}$] 
Fix an odd prime $p$. Given $H$, take $\tilde{H}$ and $(\tilde{N}_m)_{m\in \mathbb{N}}$ as in the proof of Proposition~\ref{proposition=RF(T)}. Then, $(\tilde{N}_m\times \{0\})_m$ forms a chain of a residually finite group $\tilde{H}\times  (\mathbb{Z}/2p\mathbb{Z})$. Write $\tilde{H}\times  (\mathbb{Z}/2p\mathbb{Z})$ and $(\tilde{H}/\tilde{N}_m)\times  (\mathbb{Z}/2p\mathbb{Z})$, respectively, as $L_m=L_{m}(p)$ and $L_{\infty}=L_{\infty}(p)$. Take a $3$-marking $(s_1^{(\infty)},s_2^{(\infty)},s_3^{(\infty)})$ of $L_{\infty}$, where $(s_1^{(\infty)},s_2^{(\infty)})$ corresponds to a $2$-marking of $\tilde{H}$ and $s_3^{(\infty)}$ corresponds to $[1]_{2p}\in \mathbb{Z}/2p\mathbb{Z}$. Then in $\mathcal{G}(3)$,
\[
(L_m;s_1^{(m)},s_2^{(m)},s_3^{(m)})\quad \stackrel{\mathrm{Cay}}{\longrightarrow} \quad (L_{\infty};s_1^{(\infty)},s_2^{(\infty)},s_3^{(\infty)}).
\]
Here $(s_1^{(m)},s_2^{(m)},s_3^{(m)})$ is obtained as the image of $(s_1^{(\infty)},s_2^{(\infty)},s_3^{(\infty)})$ by the quotient map $L_{\infty}\twoheadrightarrow L_m$. We may assume that $\# L_m \geq 5$, $(\# L_m)_m$ is strictly increasing on $m$, and that for every $j\in [3]$ and $m\in \mathbb{N}$, $s_j^{(m)}\ne e_{L_m}$ holds. 

Encode this convergence into symmetric groups by Lemma~\ref{lemma=SymmetricGroups} and obtain
\begin{align*}
&(\mathrm{Sym}(L_m);\chi_{s_1^{(m)}},\chi_{s_2^{(m)}},\chi_{s_3^{(m)}},\theta_{s_1^{(m)}},\theta_{s_2^{(m)}},\theta_{s_3^{(m)}})\\
\quad \stackrel{\mathrm{Cay}}{\longrightarrow} \quad& (\mathrm{Sym}(L_{\infty})\rtimes L_{\infty};\chi_{s_1^{(\infty)}},\chi_{s_2^{(\infty)}},\chi_{s_3^{(\infty)}},\theta_{s_1^{(\infty)}},\theta_{s_2^{(\infty)}},\theta_{s_3^{(\infty)}})
\end{align*}
in $\mathcal{G}(6)$. Note that for every $m\in \mathbb{N}$, $p\mid\# L_m$ and $2\mid\# L_m$. 

Fix $n\in \mathbb{N}_{\geq 3}$. For each $m\in \mathbb{N}$, we take the \textit{heart} of $p$-modular \textit{standard representation}
\[
\overline{\mathrm{std}}_{L_m,p}\colon \mathrm{Sym}(L_m)\to \mathrm{Aut}(U_m),
\]
which is defined as follows. First, let $V_m=V_m(p)$ be the zero-sum subspace of $\mathbb{F}_p^{L_m}$:
\[
V_m=\left\{(v_h)_{h\in L_m}:\ v_h \in \mathbb{F}_p\ \textrm{for every $h\in L_m$},\ \sum_{h\in L_m}v_h=0 \right\}.
\]
We warn that the permutation representation of $\mathrm{Sym}(L_m)$ on $V_m$, which is called the $p$-modular standard representation in some literature, is \textit{not} irreducible: indeed, since $p\mid\# L_m$, the constant vectors lie there. Finally, we define $U_m=U_m(p)$ to be the quotient space modulo the space of constant vectors,
\[
U_m= V_m/\mathbb{F}_p\textbf{1},
\]
on which $\mathrm{Sym}(L_m)$ acts $U_m$ by permutation of the coordinates. We claim that it is \textit{irreducible} and faithful. To show it, for the latter, observe that the alternating group $\mathrm{Alt}(L_m)$ over the set $L_m$ is simple; for the former, note that the set of vectors $\{[\delta_{h_1}-\delta_{h_2}]:h_1\ne h_2 \in L_m\}$ generates $U_m$. Set 
\[
G_m^{U_m}=\mathrm{E}(n,\overline{\mathrm{std}}_{L_m,p}(\mathbb{F}_p[\mathrm{Sym}(L_m)])),
\]
where the representation $(U_m,\overline{\mathrm{std}}_{L_m,p})$ induces 
\[
\overline{\mathrm{std}}_{L_m,p}\colon \mathbb{F}_p[\mathrm{Sym}(L_m)]\to \mathrm{End}(U_m).
\]
The key to the proof is the following: since $\overline{\mathrm{std}}_{L_m,p}$ is irreducible, it holds that
\[
\overline{\mathrm{std}}_{L_m,p}(\mathbb{F}_p[\mathrm{Sym}(L_m)]) \simeq \mathrm{Mat}_{(\# L_m-2)\times (\# L_m-2)}(\mathbb{F}_p).
\]
Indeed, every field is a splitting field for symmetric groups; see \cite[Theorem~2.1.12 and Chapter~6]{BookJamesKerber}. Therefore, we have the isomorphism
\[
G_m^{U_m}\simeq \mathrm{SL}(n(\# L_m-2),\mathbb{F}_p).
\]

To close up the proof, we consider two systems of markings of the sequence $(G_m^{U_m})_{m\in \mathbb{N}}$, as follows:
\begin{itemize}
 \item \underline{Marking corresponding to $\Lambda_1$:}\ For each $m\in \mathbb{N}$, let 
\[
t'_1{}^{(m)}=e_{1,2}^1\quad \textrm{and} \quad t'_2{}^{(m)}=\left(\begin{array}{ccccc}
0 & \cdots & \cdots & 0& -1 \\
1 & 0         &  \          &  \  & 0 \\
0 & 1         & 0         &  \  & \vdots \\
\vdots & \ddots & \ddots & \ddots & \vdots \\
0 & \cdots & 0 & 1 & 0 \\
\end{array}
\right)
\]
in $\mathrm{SL}(n(\# L_m-2),\mathbb{F}_p) (\simeq G_m^{U_m})$. These two form a $2$-marking of $G_m^{U_m}$; see \cite[Remarks~5.5 and 5.10]{MimuraSakoPartI}. Then, it is known that the sequence of $2$-marked groups $(G_m^{U_m};t'_1{}^{(m)},t'_2{}^{(m)})_{m\in \mathbb{N}}$ converges to a marked group with underlying group $\Gamma_1$ being amenable; see the last part of \cite[Remark~5.10]{MimuraSakoPartI} and \cite{LubotzkyWeiss}. 
\item \underline{Marking corresponding to $\Lambda_2$:}\ Observe that for each $m\in \mathbb{N}$, the marked group $(\mathrm{Sym}(L_m);\chi_{s_1^{(m)}},\chi_{s_2^{(m)}},\chi_{s_3^{(m)}},\theta_{s_1^{(m)}},\theta_{s_2^{(m)}},\theta_{s_3^{(m)}})$ is a marked group quotient of $(F;a_1,a_2,a_3,a_4,a_5,a_6)$. Here $F=F(p)$ be the group
\[
F=(\mathbb{Z}/2\mathbb{Z})\ast(\mathbb{Z}/2\mathbb{Z})\ast (\mathbb{Z}/2\mathbb{Z})\ast F_2\ast (\mathbb{Z}/2p\mathbb{Z})
\]
and $a_1\in \mathbb{Z}/2\mathbb{Z}$, $a_2\in \mathbb{Z}/2\mathbb{Z}$, $a_3\in \mathbb{Z}/2\mathbb{Z}$, $a_4\in \mathbb{Z}$, $a_5\in \mathbb{Z}$, and $a_6 \in \mathbb{Z}/2p\mathbb{Z}$, respectively, are cyclic generators. Here $F_2=\mathbb{Z}\ast\mathbb{Z}$ means the free group of rank $2$. The group $\mathrm{E}(n,\mathbb{F}_p[F])$ admits a $9$-marking $T=(t_1,\ldots,t_l)$; see Remark~\ref{remark=Generators} for a concrete construction. Thus, we have a system of $9$-markings of $(G_m^{U_m})_m$ by setting
\[
(t_1^{(m)},\ldots ,t_9^{(m)})=(\psi_{m}(t_1),\ldots,\psi_{m}(t_9)),
\]
where $\psi_m\colon \mathrm{E}(n,\mathbb{F}_p[F])\twoheadrightarrow G_m^{U_m}$ is the group quotient map induced by the natural ring projection $\mathbb{F}_p[F]\twoheadrightarrow \overline{\mathrm{std}}_{L_m,p}(\mathbb{F}_p[\mathrm{Sym}(L_m)])$. 
\end{itemize}

Now set the compact group $K^{(H)}$ as 
\begin{align*}
K^{(H)}=\prod_{m\in \mathbb{N}}G_m^{U_m}\quad \left(\simeq \prod_{m\in \mathbb{N}}\mathrm{SL}(n(\# L_m-2),\mathbb{F}_p)\right).
\end{align*}
Hence $K^{(H)}$ is of the form 
\[
K^{(H)}=\prod_{m\in \mathbb{N}}\mathrm{SL}(nl_m,\mathbb{F}_p),
\]
where we set $(l_m)_{m\in \in \mathbb{N}}$ by for every $m\in \mathbb{N}$, $l_m=\# (L_m(p))-2(=2p\# (\tilde{H}_m/\tilde{N}_m)-2)$. 
We define $\Lambda_1$ and $\Lambda_2$, respectively, by
\begin{align*}
&\Lambda_1=\textrm{$($the underlying group of $\Delta_{m\in \mathbb{N}}((G_m^{U_m};t'_1{}^{(m)},t'_2{}^{(m)}))$$)$},\\
\textrm{and} \quad &\Lambda_2=\textrm{$($the underlying group of $\Delta_{m\in \mathbb{N}}((G_m^{U_m};t_1^{(m)},t_2^{(m)},\ldots ,t_9^{(m)}))$$)$}.
\end{align*}
We then claim that these $K^{(H)}$, $\Lambda_1$ and $\Lambda_2$ fulfill all of the assertions as in Theorem~\ref{theorem=Main}. Indeed, by Lemma~\ref{lemma=Goursat}, $\Lambda_1$ and $\Lambda_2$ are both dense in $K^{(H)}$. Since $\Gamma_1$ above is amenable, it follows from $(\ast)$ that assertion $(i)$ is satisfied; recall that amenability is closed under taking group extensions. Assertion $(ii)$ follows because $\Delta_{m\in \mathbb{N}}((G_m^{U_m};t_1^{(m)},\ldots ,t_9^{(m)}))$ is a marked group quotient of $(\mathrm{E}(n,\mathbb{F}_p[F]);t_1,\ldots ,t_9)$. Here recall Theorem~\ref{theorem=ErshovJaikin}; recall also that property $(\mathrm{T})$ passes to group quotients. To show $(iii)$, let $\Gamma_2$  be the underlying group of the Cayley limit of $((G_m^{U_m};t_1^{(m)},\ldots ,t_9^{(m)}))_m$. First recall that $((L_m;\theta_{s_1^{(m)}},\theta_{s_2^{(m)}},\theta_{s_3^{(m)}}))_m$ converges to $(L_{\infty};\theta_{s_1^{(\infty)}},\theta_{s_2^{(\infty)}},\theta_{s_3^{(\infty)}})$ in $\mathcal{G}(3)$. Observe that the convergence above comes from the chain, which we chose at the beginning of the proof, for the residually finite group $L_{\infty}$. Combine it with Lemma~\ref{lemma=Order2}; recall in addition that $\overline{\mathrm{std}}_{L_m,p}$ is faithful. Then, it may be verified that $\Gamma_2$ admits an isomorphic copy of $\overline{H}$ such that it is  \textit{isomorphically lifted} under the lift $\Lambda_2\twoheadrightarrow \Gamma_2$. Therefore, we conclude that
\[
\Lambda_2\geqslant \overline{H}\quad (\geqslant H),
\]
as desired. For more details of the argument on assertion $(iii)$, consult Lemma~\ref{lemma=LEFsubgroups} and \cite[Subsection~4.4]{MimuraRF}. 
\end{proof}

\section{Proof of Theorem~$\ref{theorem=Common}$}\label{section=Common}

The construction of $K^{(H)}$ in Theorem~\ref{theorem=Main} \textit{does} depend on the choice of $H$; more precisely, it depends on the sequence $(l_m)_m=(2p\# (\tilde{H}_m/\tilde{N}_m)-2)_m$. Hence, if we decompose
\begin{align*}
\mathcal{K}&=K^{(H)}\times \mathcal{K}_0^{(H)}\\
&=\left(\prod_{m\in \mathbb{N}}\mathrm{SL}(nl_m,\mathbb{F}_p)\right)\times \left(\prod_{i\in \mathbb{N}\setminus \{l_m:\,m\in \mathbb{N}\}}\mathrm{SL}(ni,\mathbb{F}_p)\right),
\end{align*}
then what remains is to deal with the latter component $\mathcal{K}_0^{(H)}$.

The key to the proof of Theorem~$\ref{theorem=Common}$ is the following: let $p$ be an odd prime and $n\in \mathbb{N}_{\geq 4}$ be even. Then for \textit{every} $i\in \mathbb{N}_{\geq 1}$, $\mathrm{SL}(ni,\mathbb{F}_p)\simeq \mathrm{E}(n,\mathrm{Mat}_{i\times i}(\mathbb{F}_p))$ is a group quotient of $\mathrm{E}(n,\mathbb{F}_p \langle y,z \rangle)$, where $\mathbb{F}_p \langle y,z \rangle$ denotes the non-commutative polynomial ring over $\mathbb{F}_p$ with indeterminates $y$ and $z$. Indeed, the map that sends $y$ to $e_{1,2}^1\in \mathrm{Mat}_{i\times i}(\mathbb{F}_p)$ and $z$ to the matrix in $\mathrm{Mat}_{i\times i}(\mathbb{F}_p)$ representing a cyclic permutation induces a group quotient map $\mu_i\colon \mathrm{E}(n,\mathbb{F}_p \langle y,z \rangle)\twoheadrightarrow \mathrm{SL}(ni,\mathbb{F}_p)$. In a manner similar to one in Remark~\ref{remark=Generators}, it follows that $\mathrm{E}(n,\mathbb{F}_p \langle y,z \rangle)$ is generated by four elements 
\[
b_1=e_{1,2}^1,\ b_2=e_{1,2}^{y},\ b_3=e_{1,2}^{z}\quad  \textrm{and} \quad 
b_4=\left(\begin{array}{ccccc}
0 & \cdots & \cdots & 0& -1 \\
1 & 0         &  \          &  \  & 0 \\
0 & 1         & 0         &  \  & \vdots \\
\vdots & \ddots & \ddots & \ddots & \vdots \\
0 & \cdots & 0 & 1 & 0 \\
\end{array}
\right)
\]
in $\mathrm{E}(n,\mathbb{F}_p \langle y,z \rangle)$. Let $(\beta_1^{(i)},\beta_2^{(i)},\beta_3^{(i)},\beta_4^{(i)})$ be $(\mu_i(b_1),\mu_i(b_2),\mu_i(b_3),\mu_i(b_4))$. Then, for each $i\in \mathbb{N}$, $(\beta_1^{(i)},\beta_2^{(i)},\beta_3^{(i)},\beta_4^{(i)})$ is a $4$-marking of $\mathrm{SL}(ni,\mathbb{F}_p)$. By Theorem~\ref{theorem=ErshovJaikin}, the underlying group $\Sigma_0$ of $\Delta_{i\in \mathbb{N}}((\mathrm{SL}(ni,\mathbb{F}_p);\beta_1^{(i)},\beta_2^{(i)},\beta_3^{(i)},\beta_4^{(i)}))$ has property $(\mathrm{T})$.  For an odd $n\in \mathbb{N}_{\geq 3}$, we replace $b_4$ above with $b_4'$ as in the proof below.

Ershov and Jaikin-Zapirain \cite{ErshovJaikinZapirain} constructed the first counterexample to Conjecture~\ref{conjecture=LubotzkyWeiss} in the manner above.

\begin{proof}[Proof of Theorem~$\ref{theorem=Common}$]
Fix $p$ and $n$ as in the statement of Theorem~$\ref{theorem=Common}$. For each $i\in \mathbb{N}_{\geq 1}$, let $Q_i(=Q_i(n,p))=\mathrm{SL}(ni,\mathbb{F}_p)$. Observe that for every $p$ and $n$, the sequence $(Q_i)_{i\in \mathbb{N}}$ fulfills the condition of Lemma~\ref{lemma=Goursat}. 

First, we discuss the case where $n$ is even. Note that then for every $i\in \mathbb{N}_{\geq 1}$, $ni$ is even. Let 
\[
\alpha_1{}^{(i)}=e_{1,2}^1\quad \textrm{and} \quad \alpha_2{}^{(i)}=\left(\begin{array}{ccccc}
0 & \cdots & \cdots & 0& -1 \\
1 & 0         &  \          &  \  & 0 \\
0 & 1         & 0         &  \  & \vdots \\
\vdots & \ddots & \ddots & \ddots & \vdots \\
0 & \cdots & 0 & 1 & 0 \\
\end{array}
\right)
\]
in $\mathrm{SL}(ni,\mathbb{F}_p)(=Q_i)$. For the same reasoning as in the proof of Theorem~\ref{theorem=Main}, the Cayley limit of $(Q_i;\alpha_1^{(i)},\alpha_2^{(i)})$ (as $i\to \infty$) is amenable. By taking the underlying group of the diagonal product $\Delta_{i\in \mathbb{N}_{\geq 1}}((Q_i;\alpha_1^{(i)},\alpha_2^{(i)}))$, we obtain $\Sigma_1$. Next, we will construct $\Sigma_2=\Sigma_2^{(H)}$ for a given countable and residually finite $H$. For such $H$, take $(L_m)_{m\in \mathbb{N}}=(L_m(p))_m$ as in the proof of Theorem~\ref{theorem=Main}. Let $(l_m)_{m\in \mathbb{N}}$ be the sequence given by $l_m=\# L_m-2$. Then by the proof of Theorem~\ref{theorem=Main}, for each $m\in \mathbb{N}$, we have the marking
\[
(Q_{l_m};t_1^{(m)},\ldots ,t_9^{(m)})
\]
such that the underlying group $\Lambda_2=\Lambda_2^{(H)}$ of $\Delta_{m\in \mathbb{N}}((Q_{l_m};t_1^{(m)},\ldots ,t_9^{(m)}))$ has property $(\mathrm{T})$ and that $\Lambda_2$ contains an isomorphic copy of $H$. Let
\[
\mathcal{K}_0=\mathcal{K}_0^{(H)}=\prod_{i\in \mathbb{N}\setminus \{l_m:\,m\in \mathbb{N}\}}Q_i.
\]
Now we make use of the $4$-marking $(\beta_1^{(i)},\beta_2^{(i)},\beta_3^{(i)},\beta_4^{(i)})$ of $G_i$, described in the beginning of this section. Then, the underlying group $\Sigma_0^{(H)}$ of 
\[
\Delta_{i\in \mathbb{N}\setminus \{l_m:\,m\in \mathbb{N}\}}((\mathrm{SL}(ni,\mathbb{F}_p);\beta_1^{(i)},\beta_2^{(i)},\beta_3^{(i)},\beta_4^{(i)}))
\]
has property $(\mathrm{T})$ as well.

Finally, for each $i\in \mathbb{N}$, consider the following $13$-marked group:
\begin{align*}
&(Q_i;\tau_1^{(i)},\tau_2^{(i)},\ldots ,\tau_{13}^{(i)})\\
=&
\left\{\begin{array}{l}
(Q_i;t_1^{(m)},t_2^{(m)},t_3^{(m)},t_4^{(m)},t_5^{(m)},t_6^{(m)},t_7^{(m)},t_8^{(m)},t_9^{(m)},e_{Q_i},e_{Q_i},e_{Q_i},e_{Q_i}),\\
(Q_i;e_{Q_i},e_{Q_i},e_{Q_i},e_{Q_i},e_{Q_i},e_{Q_i},e_{Q_i},e_{Q_i},e_{Q_i},\beta_1^{(i)},\beta_2^{(i)},\beta_3^{(i)},\beta_4^{(i)}).
\end{array}
\right.
\end{align*}
Here the first case applies if there exists (unique) $m\in \mathbb{N}$ such that $i=l_m$; otherwise the second case applies. It is straightforward to see that the underlying group $\Sigma_2^{(H)}$ of $\Delta_{i\in \mathbb{N}}((Q_i;\tau_1^{(i)},\ldots ,\tau_{13}^{(i)}))$ is isomorphic to $\Lambda_2^{(H)}\times\Sigma_0^{(H)}$. Hence the group $\Sigma_2^{(H)}$ has property $(\mathrm{T})$ and it contains an isomorphic copy of $H$. It is dense in $\mathcal{K}$ by Lemma~\ref{lemma=Goursat}. This ends the proof for the case where $n$ is even. 

Secondly, we deal with the case where $n$ is odd. In this case, decompose $\mathcal{K}$ as 
\[
\mathcal{K}=\prod_{i\in \mathbb{N}_{\geq 1,\mathrm{even}}}Q_{i}\times \prod_{i\in \mathbb{N}_{\mathrm{odd}}}Q_i,
\]
where we define $\mathbb{N}_{\geq 1,\mathrm{even}}=\mathbb{N}_{\geq 1}\cap 2\mathbb{Z}$ and $\mathbb{N}_{\mathrm{odd}}=\mathbb{N}\cap (2\mathbb{Z}+1)$. For $i\in \mathbb{N}_{\mathrm{odd}}$, let
\[
\alpha_2'{}^{(i)}=\left(\begin{array}{ccccc}
0 & \cdots & \cdots & 0& 1 \\
1 & 0         &  \          &  \  & 0 \\
0 & 1         & 0         &  \  & \vdots \\
\vdots & \ddots & \ddots & \ddots & \vdots \\
0 & \cdots & 0 & 1 & 0 \\
\end{array}
\right)
\]
in $\mathrm{SL}(ni,\mathbb{F}_p)(=Q_i)$. Define a $4$-marked group $(Q_i;\sigma_1^{(i)},\sigma_2^{(i)},\sigma_3^{(i)},\sigma_4^{(i)})$ by
\begin{align*}
(Q_i;\sigma_1^{(i)},\sigma_2^{(i)},\sigma_3^{(i)},\sigma_4^{(i)})
=
\left\{\begin{array}{cl}
(Q_i;\alpha_1^{(i)},\alpha_2^{(i)},e_{Q_i},e_{Q_i}),& \textrm{if\ }i\in \mathbb{N}_{\geq 1,\mathrm{even}},\\
(Q_i;e_{Q_i},e_{Q_i},\alpha_1^{(i)},\alpha_2'{}^{(i)}),& \textrm{if\ }i\in \mathbb{N}_{\mathrm{odd}}.
\end{array}
\right.
\end{align*}
Then the underlying group $\Sigma_1$ of $\Delta_{i\in \mathbb{N}}((Q_i;\sigma_1^{(i)},\sigma_2^{(i)},\sigma_3^{(i)},\sigma_4^{(i)}))$ is amenable and dense in $\mathcal{K}$. To obtain $\Sigma_2^{(H)}$, we may take the same construction as one in the case of an even $n$ with the following modification: replace $b_4$ with 
\[
b_4'=\left(\begin{array}{ccccc}
0 & \cdots & \cdots & 0& 1 \\
1 & 0         &  \          &  \  & 0 \\
0 & 1         & 0         &  \  & \vdots \\
\vdots & \ddots & \ddots & \ddots & \vdots \\
0 & \cdots & 0 & 1 & 0 \\
\end{array}
\right)
\] in $\mathrm{E}(n,\mathbb{F}_p\langle y,z\rangle)$ and let $\beta_4^{(i)}=\mu_i(\beta_4')$, instead of setting $\beta_4^{(i)}=\mu_i(\beta_4)$. 
\end{proof}

\section{Remarks}\label{section=Remarks}
\begin{remark}\label{remark=SQ}
In Proposition~\ref{proposition=RF(T)}, it is rather standard to embed a given $H$ into a group with property $(\mathrm{T})$ in the following way. Take an infinite hyperbolic group $\tilde{C}$ with property $(\mathrm{T})$. Then by SQ-universality of $\tilde{C}$ \cite{Olshanskii}, there exists a  group quotient $C$ of $\tilde{C}$ such that $H$ embeds into $C$. However, even if $\tilde{C}$ is residually finite, the procedure $\tilde{C}\twoheadrightarrow C$ may  spoil the residual finiteness.
\end{remark}

\begin{remark}\label{remark=Osajda}
The original construction of residually finite non-exact groups by Osajda \cite{OsajdaRF} made essential use of the work of Wise \cite{wise} and Agol \cite{agol} on virtually special groups to obtain residual finiteness; these groups are direct limits of certain virtually special groups. All virtually special groups are known to have the \textit{Haagerup property}, which may be seen as a strong negation of property $(\mathrm{T})$ for infinite countable groups. It then follows that the originally constructed non-exact groups above (or their variant in \cite[Remark~3.3]{MimuraRF}) \textit{never} have property $(\mathrm{T})$. 
\end{remark}

\begin{remark}\label{remark=Generators}
Let $p$ be a prime and $F=(\mathbb{Z}/2\mathbb{Z})\ast(\mathbb{Z}/2\mathbb{Z})\ast (\mathbb{Z}/2\mathbb{Z})\ast F_2\ast (\mathbb{Z}/2p\mathbb{Z})
$ as in the proof of Theorem~\ref{theorem=Main}. Let $R=\mathbb{F}_p[F]$: it is isomorphic to the non-commutative ring $\mathbb{F}_p\langle x_1,x_2,x_3,x_4^{\pm },x_5^{\pm},x_6\rangle$ subject to the relation $x_1^2=x_2^2=x_3^2=x_6^{2p}=1$. Let $\Gamma=\mathrm{E}(n,R)$. We provide a concrete $9$-marking of $\Gamma$, which was used in the proof of Theorem~\ref{theorem=Main}.

If $n\in \mathbb{N}_{\geq 3}$ above is odd, then $\Gamma$ is generated by $e_{1,2}^{x}$, $x\in \{x_1,x_2,x_3,x_4,x_4^{-1},x_5,x_5^{-1},x_6\}$, and by the matrix $\beta=\beta_4'$ as in the proof of Theorem~\ref{theorem=Common}. Indeed, for $r\in R$, $\beta e_{1,2}^{r}\beta^{-1}=e_{2,3}^{r}$ holds. Observe the following commutator relation: for $r_1,r_2\in R$,
\[
[e_{i,j}^{r_1},e_{j,k}^{r_2}]=e_{i,k}^{r_1r_2} \quad \textrm{for every }i\ne j\ne k\ne i. \tag{$\#$}
\]
Here our convention of commutators is $[g_1,g_2]=g_1^{-1}g_2^{-1}g_1g_2$. 
It follows from $(\#)$ that $[e_{1,2}^{x_4},e_{2,3}^{(x_4^{-1})}]=e_{1,3}^1$. Then by taking conjugations of $e_{1,3}^1$ by powers of $\beta$ and by  $(\#)$, we may obtain all elements of the form $e_{i,j}^1$, $i\ne j\in [n]$. Again by $(\#)$, we may express every matrix of the form $e_{i,j}^r$, $i\ne j\in [n]$ and $r\in R$, as some product of the nine elements above; note also that $e_{i,j}^{r_1}e_{i,j}^{r_2}=e_{i,j}^{r_1+r_2}$.

For an even number $n\in \mathbb{N}_{\geq 3}$, we may let $\beta=\beta_4$ as in the proof of Theorem~\ref{theorem=Common}, instead of setting $\beta=\beta_4'$, and obtain a $9$-marking of $\Gamma$ as well.
\end{remark}

\begin{remark}\label{remark=LEF}
One of the key to the proof of Theorem~\ref{theorem=Main} is   to extend the framework from residually finite groups to LEF groups. Indeed, we employ Lemma~\ref{lemma=SymmetricGroups} to deal with the two difficulties explained in the former part of Section~\ref{section=TheProof}. Here, even if the original convergence of $(G_m)_m$ to $\Gamma$ (with respect to certain markings) comes from a chain associated with residual finiteness, the convergent sequence $(\mathrm{Sym}(G_m))_m$ to $\mathrm{Sym}_{<\aleph_0}(\Gamma)\rtimes \Gamma$  is \textit{merely} a LEF approximation. More precisely, the group $\mathrm{Sym}_{<\aleph_0}(\Gamma)\rtimes \Gamma$ is \textit{never} residually finite. This is because it contains $\mathrm{Alt}(\Gamma)$, the increasing union of finite alternating subgroups, which is an infinite simple group.
\end{remark}

\begin{remark}\label{remark=LFNF-lifts}
By $(\ast)$ and \cite[Remark~5.10]{MimuraSakoPartI}, $\Lambda_1$ as in the proof of Theorem~\ref{theorem=Main} admits the following short exact sequence:
\[
1\quad \longrightarrow \quad N \quad \longrightarrow \quad  \Lambda_1 \quad \longrightarrow \quad N(\mathbb{Z},\mathbb{F}_p)\rtimes \mathbb{Z} \quad \longrightarrow \quad 1.
\]
Here $N$ above is \textit{locally fully normally finite}, that means, for every non-empty finite subset $A\subseteq N$, the normal closure $\llangle A\rrangle_{\Lambda_1}$ \textit{in $\Lambda_1$} is finite; see also \cite[Definition~3.1 and Remark~5.6]{MimuraRF}. The group $N(\mathbb{Z},\mathbb{F}_p)$ is the increasing union of 
\[
(N(2m+1,\mathbb{F}_p)\simeq)\quad N([-m,m]\cap \mathbb{Z},\mathbb{F}_p)=\langle e_{i,j}^r: i,j\in [-m,m]\cap \mathbb{Z},\ i>j, \ r\in \mathbb{F}_p\rangle
\]
over $m\in \mathbb{N}$; here we regard $\mathbb{Z}$ and $[-m,m]\cap \mathbb{Z}$ as sets of indices, and embed naturally $N(2m+1,\mathbb{F}_p)$ into $N(\mathbb{Z},\mathbb{F}_p)$. To construct $N(\mathbb{Z},\mathbb{F}_p)\rtimes \mathbb{Z}$, we let $\mathbb{Z}$ on the right side acts on $N(\mathbb{Z},\mathbb{F}_p)$ by the translation on the index set $\mathbb{Z}$. Note that each $N(2m+1,\mathbb{F}_p)$ is a $p$-group. In particular, $\Lambda_1$ above admits a surjection onto $\mathbb{Z}$ with locally finite kernel; it has \textit{asymptotic dimension} (see \cite[Chapter~2]{bookNowakYu}) $1$.

In the construction of $\Lambda_1$ as in the proof of Theorem~\ref{theorem=Main}, we may replace $p$ with a sequence $(p_m)_m$ indexed by $m\in \mathbb{N}$ to ensure that $p_m\nmid\#(\tilde{H}/\tilde{N}_m)$. In that case, the $p_m$-modular standard representation of $\mathrm{Sym}(\tilde{H}/\tilde{N}_m)$ is irreducible and we do not need to take a quotient vector space. However, the resulting $\Lambda_1$ in this case is obtained by replacement of $N(\mathbb{Z},\mathbb{F}_p)$ with $N(\mathbb{Z},\mathbb{Z})$; this new $\Lambda_1$ contains an isometric copy of $\mathbb{Z}^l$ for every $l\in \mathbb{N}$. It is a quite ``huge'' group, compared with the original $\Lambda_1$ constructed in the proof of Theorem~\ref{theorem=Main}.

Similarly, if $n$ is even in Theorem~\ref{theorem=Common}, then $\Sigma_1$ as in the proof of Theorem~\ref{theorem=Common} admits the short exact sequence same as one in above in the present remark.
\end{remark}

\begin{remark}\label{remark=Upgrading}
By \cite{MimuraUpgrading}, if we take $n$ from $\mathbb{N}_{\geq 4}$, then the group $\Lambda_2$ as in the proof of Theorem~\ref{theorem=Main}, in fact, enjoys the fixed point property $(\mathrm{F}_{\mathcal{B}_{L_q}})$ with respect to $L_q$-spaces for \textit{all} $q\in (1,\infty)$. See also \cite[Subsection~5.2.2 and Remark~5.7]{Oppenheim} for other fixed point properties for $\Lambda_2$ above when the prime $p$ in the proof of Theorem~\ref{theorem=Main} is sufficiently large. 
\end{remark}

\begin{remark}\label{remark=2-generated}
If we are allowed to ``enlarge'' the amenable group $\Lambda_1$ as in the proof of Theorem~\ref{theorem=Main}, in the sense of Remark~\ref{remark=LFNF-lifts}, then we have the following result.

\begin{theorem}\label{theorem=2-generated}
Given a countable residually finite group $H$, there exist a compact group $\tilde{K}=\tilde{K}^{(H)}$ and $\tilde{w}_1,\tilde{w}_2,\tilde{u} \in \tilde{K}$ such that the following hold true: for $\tilde{\Lambda}_1=\langle \tilde{w}_1,\tilde{u}\rangle$ and $\tilde{\Lambda}_2=\langle \tilde{w}_2,\tilde{u}\rangle$, \begin{itemize}
  \item $\tilde{\Lambda}_1$ and $\tilde{\Lambda}_2$ are both dense in $\tilde{K}$;
  \item $\tilde{\Lambda}_1$ is amenable $($with finite asymptotic dimension$)$;
  \item $\tilde{\Lambda}_2$ has property $(\mathrm{T})$; and 
  \item $\tilde{\Lambda}_2$ contains  an isomorphic copy of $H$.
\end{itemize}
\end{theorem}

The construction uses \textit{wreath products} (with finitely many coordinates): for a (possibly topological) group $\mathcal{L}$ and for a \textit{finite} group $A$, set 
\[
\mathcal{L}\wr A=\left(\bigoplus_{A} \mathcal{L}\right)\rtimes A, 
\]
which is endowed with the natural topology induced by that on $\mathcal{L}^A \times A$. Here the $A$-action in the right hand side of the equality above is given by permutation of coordinates by right multiplication. An element in $\bigoplus_{A} \mathcal{L}$ is naturally identified with a map $f\colon A\to \mathcal{L}$; we write the constant map taking the value in $\{e_{\mathcal{L}}\}$ as $\mathbf{e}$.

\begin{proof}
Fix an odd prime $p$. Fix odd $n\in \mathbb{N}_{\geq 5}\setminus p\mathbb{Z}$. Take the construction of $(G_m^{U_m})_{m\in \mathbb{N}}$, $(t'_1{}^{(m)},t'_2{}^{(m)})$, $(\mathrm{E}(n,\mathbb{F}_p[F]);t_1,\ldots ,t_9)$, $\psi_m\colon \mathrm{E}(n,\mathbb{F}_p[F])\twoheadrightarrow G_m^{U_m}$, $\Lambda_1$, $\Lambda_2$ and   $K=K^{(H)}$ as in the proof of Theorem~\ref{theorem=Main}. 

We claim that for each $m\in \mathbb{N}$, for every $j\in [9]$, $t_j\in \mathrm{E}(n,\mathbb{F}_p[F])$ may be written as a single commutator. Indeed, by $(\#)$, for every $r\in R$, it holds that $e_{1,2}^r=[e_{1,3}^r,e_{3,2}^1]$;
for the permutation matrix $\beta$, make use of the fact that every element in $\mathrm{Alt}(n)$ is a single commutator  (\cite{Ore}). Hence, for each $j\in [9]$, we have $c_j,d_j\in \mathrm{E}(n,\mathbb{F}_p[F])$ such that $[c_j,d_j]=t_j$. 

For every $m\in \mathbb{N}$, set $w_1^{(m)}, w_2^{(m)},u^{(m)} \in G_m^{U_m}\wr (\mathbb{Z}/2^{20}\mathbb{Z})$ as 
\[
w_1^{(m)}=(f_1^{(m)},0),\quad w_2^{(m)}=(f_2^{(m)},0) \quad \textrm{and} \quad u=(\mathbf{e},1).
\]
Here $f_1^{(m)},f_2^{(m)}\in \bigoplus_{\mathbb{Z}/2^{20}\mathbb{Z}} G_m^{U_m}$ are, respectively, defined by for every $l\in \mathbb{Z}/2^{20}\mathbb{Z}$,
\[
f_1^{(m)}(l)=\left\{\begin{array}{cl}
t'_1{}^{(m)}, & \textrm{$l=0$},\\
t'_2{}^{(m)}, & \textrm{$l=1$},\\
e_{G_m^{U_m}} , & \mathrm{otherwise}
\end{array}\right. \quad \textrm{and} \quad 
f_2^{(m)}(l)=\left\{\begin{array}{cl}
\psi_m(c_j), & \textrm{$l=2^j$,\ $j\in [9]$},\\
\psi_m(d_j), & \textrm{$l=2^{j+9}$,\ $j\in [9]$},\\
e_{G_m^{U_m}} , & \mathrm{otherwise}
\end{array}\right. .
\]
We claim that $(w_2^{(m)},u^{(m)})$ is a $2$-marking of $G_m^{U_m}\wr (\mathbb{Z}/2^{20}\mathbb{Z})$. Indeed, by commutator calculus, it follows that for every $j\in [9]$,
\[
[u^{2^j}w_2^{(m)}u^{-2^j},u^{2^{j+9}}w_2^{(m)}u^{-2^{j+9}}]=(g_j^{(m)},0),
\]
where for $l\in \mathbb{Z}/2^{20}\mathbb{Z}$,
\[
g_j^{(m)}(l)=\left\{\begin{array}{cl}
[\psi_m(c_j),\psi_m(d_j)]\quad(=\psi_m(t_j)), & \textrm{$l=0$},\\
e_{G_m^{U_m}} , & \mathrm{otherwise}.
\end{array}\right. 
\]
This is a Hall-type argument \cite[1.5]{Hall}; see \cite[the proof of Lemma~4.8]{MimuraRF} for more details. 

It is easier to see that $(w_1^{(m)},u^{(m)})$ is a $2$-marking of $G_m^{U_m}\wr (\mathbb{Z}/2^{20}\mathbb{Z})$. Indeed, the order of $t'_1{}^{(m)}$ is $p$ and that of $t'_2{}^{(m)}$ is $n(\# L_m-2)$. Recall that $p\mid \# L_m$, $p\ne 2$ and that $n\not \in p\mathbb{N}$. It then follows that these two numbers, $p$ and $n(\# L_m-2)$, are coprime. Hence we may extract, respectively, $t'_1{}^{(m)}$ and $t'_2{}^{(m)}$ from $f_1^{(m)}$ by taking, respectively, appropriate powers of $w_1^{(m)}$.

Finally, set $\tilde{K}=\tilde{K}^{(H)}$ by
\[
\tilde{K}=K\wr (\mathbb{Z}/2^{20}\mathbb{Z})\quad \left(= (\prod_{m\in \mathbb{N}}(\bigoplus_{\mathbb{Z}/2^{20}\mathbb{Z}}G_m^{U_m}))\rtimes (\mathbb{Z}/2^{20}\mathbb{Z})\right).
\]
Here we have the identification of the second equality above by letting $\mathbb{Z}/2^{20}\mathbb{Z}$ act on $\prod_{m\in \mathbb{N}}(\bigoplus_{\mathbb{Z}/2^{20}\mathbb{Z}}G_m^{U_m})$ coordinatewise for each $m\in \mathbb{N}$.
Let 
\begin{align*}
\tilde{w}_1&=((f_1^{(0)},f_1^{(1)},\ldots ,f_1^{(m)},\ldots ),0) ,\\
\tilde{w}_2&=((f_2^{(0)},f_2^{(1)},\ldots ,f_2^{(m)},\ldots ),0) ,\\
\textrm{and}\quad \quad\tilde{u}&=(\mathbf{e},1) 
\end{align*}
be elements in $\tilde{K}$ (via the identification above for $\tilde{w}_1$ and $\tilde{w}_2$). Then, in a similar way to one in the proof of Theorem~\ref{theorem=Main}, it may be proved that these $\tilde{K}$, $\tilde{w}_1$, $\tilde{w}_2$ and $\tilde{u}$ fulfill all of the four assertions as in Theorem~\ref{theorem=2-generated}. More precisely, to prove the third assertion (property $(\mathrm{T})$ for $\tilde{\Lambda}_2$), observe that $\tilde{\Lambda}_2$ is a group quotient of $\mathrm{E}(n,\mathbb{F}_p[F])\wr (\mathbb{Z}/2^{20}\mathbb{Z})$. Since property $(\mathrm{T})$ is closed under taking extensions, the wreath product above has property $(\mathrm{T})$; recall that finite groups have property $(\mathrm{T})$. To see that $\tilde{\Lambda}_1$ has finite asymptotic dimension, observe that $\tilde{\Lambda}_1$ is a subgroup of $\Lambda_1\wr (\mathbb{Z}/2^{20}\mathbb{Z})$, where $\Lambda_1$ is as in the proof of Theorem~\ref{theorem=Main}; recall Remark~\ref{remark=LFNF-lifts}. (Hence, in our construction $\tilde{\Lambda}_1$ is elementary amenable, as well as the group $\Lambda_1$ above.)
\end{proof}
\end{remark}

The following strengthening of Theorem~\ref{theorem=Common} may be derived from a similar argument to one in the proof of Theorem~\ref{theorem=2-generated}. We leave the proof of it to the reader. More precisely, to apply a Hall-type argument to $\tilde{\Sigma}_1$, note that for every $m\in \mathbb{N}$, $l_m$ is coprime to $p$. Here $p$ is a fixed odd prime and $(l_m)_{m\in \mathbb{N}}$, associated with $p$  and a countable residually finite group $H$, is as in the proof of Theorem~\ref{theorem=Main}.

\begin{theorem}[Variety of \textit{$2$-generated} dense subgroups of a \textit{common} compact group]\label{theorem=2-generatedCommon}
Let $p$ be an odd prime. Let $n\in \mathbb{N}_{\geq 6}$ be an even number such that $p \nmid n$. Then, there exist a compact group
\[
\tilde{\mathcal{K}}(=\tilde{\mathcal{K}}_{n,p})=\left( \prod_{i\in \mathbb{N}_{\geq 1}\setminus p\mathbb{Z}} \mathrm{SL}(ni,\mathbb{F}_p)\right)\wr (\mathbb{Z}/2^{30}\mathbb{Z})
\]
and $\tilde{\omega}_{1},\tilde{\upsilon}\in \tilde{\mathcal{K}}$ such that the following hold ture:
\begin{enumerate}[$(1)$]
  \item The group $\tilde{\Sigma}_1=\langle \tilde{\omega}_{1},\tilde{\upsilon}\rangle$ is dense in $\tilde{\mathcal{K}}$. It is amenable and it has finite asymptotic dimension.
  \item For every countable and residually finite group $H$, there exists $\tilde{\omega}_{2}=\tilde{\omega}_{2}^{(H)}$ such that the group $\tilde{\Sigma}_2=\tilde{\Sigma}_2^{(H)}=\langle \tilde{\omega}_{2}^{(H)},\tilde{\upsilon}\rangle$ fulfills the following. The group $\tilde{\Sigma}_2^{(H)}$ is dense in $\tilde{\mathcal{K}}$, it has property $(\mathrm{T})$ and it contains an isomorphic copy of $H$.
\end{enumerate}
\end{theorem}

\section*{Acknowledgments}
Part of this work has been done during the  two-year stay of the author in the \'{E}cole Polytechnique F\'{e}d\'{e}rale de Lausanne supported by Grant-in-Aid for JSPS Oversea Research Fellowships. The author wishes to express his gratitude to Professor Nicolas Monod and Mrs. Marcia Gouffon at the EPFL for their hospitality and help on his stay. He is grateful to V\'{i}ctor Hugo Ya\~{n}ez, Motoko Kato, Shin-ichi Oguni, Tatsuya Ohshita, Yoshinori Yamasaki and  Takamitsu Yamauchi  for discussions during his stay in Ehime University. He thanks Akihiro Munemasa and Sven Raum for comments. The author is supported in part by JSPS KAKENHI Grant Number JP17H04822.

\bibliographystyle{amsalpha}
\bibliography{rft_f.bib}

\end{document}